\documentclass{article}

\usepackage[english]{babel}
\usepackage[letterpaper,top=2cm,bottom=2cm,left=3cm,right=3cm,marginparwidth=1.75cm]{geometry}
\usepackage{indentfirst}

\usepackage{graphicx}
\usepackage[colorlinks=true, allcolors=blue]{hyperref}
\usepackage{amssymb,amsmath,amsthm,amsfonts,xcolor,enumerate,hyperref,comment,longtable,cleveref,amscd}

\newtheorem{theorem}{Theorem}
\newtheorem{lemma}[theorem]{Lemma}
\newtheorem{proposition}[theorem]{Proposition}

\newtheorem{definition}[theorem]{Definition}
\newtheorem{corollary}[theorem]{Corollary}

\title{}
\author{}
\date{}

\begin{document}
\maketitle

\begin{Huge}
    \begin{center}
    Transposed $\delta$-Poisson algebra structures on null-filiform associative algebras
\end{center}
\end{Huge}

\begin{center}

 {\bf  Nigora Daukeyeva\footnote{Chirchik State Pedagogical University, Chirchik, Uzbekistan; \  
 daukeyeva.cspu@gmail.com},  
Maqpal Eraliyeva\footnote{
Chirchik State Pedagogical University, Chirchik, Uzbekistan; \  eraliyevamaqpal@gmail.com},
Feruza Toshtemirova\footnote{Chirchik State Pedagogical University, Chirchik, Uzbekistan; \ feruzaisrailova45@gmail.com}
}

\end{center}

\begin{abstract}
In this paper, we consider transposed $\delta$-Poisson algebras, which are a generalization of transposed $\delta$-Poisson algebras. In particular, we describe all transposed $\delta$-Poisson algebras of associative null-filiform algebras. It can be seen that these algebras are characterized by the roots of the polynomial $\delta^3 - 3\delta^2 + 2\delta$. A complete classification of transposed $\delta$-Poisson algebras corresponding to each value of the parameter $\delta$ is provided. Furthermore, we construct all $\delta$-Poisson algebra structures on null-filiform associative algebras, and show that they are trivial $\delta$-Poisson algebras.
\end{abstract}



\section{Introduction}
Poisson algebras originated in the study of Poisson geometry during the 1970s and have since appeared in a wide range of mathematical and physical disciplines, including Poisson manifolds, algebraic geometry, operads, quantization theory, quantum groups, and both classical and quantum mechanics. More recently, the notion of transposed Poisson algebras has been introduced in \cite{Bai}, providing a dual perspective on Poisson structures and leading to novel algebraic frameworks. This concept has found applications in various algebraic structures, including Novikov-Poisson algebras and $3$-Lie algebras \cite{bfk22}. In \cite{jawo, kk21, Kh25, YYZ07}, Poisson structures on canonical algebras and the finitary incidence algebras of an arbitrary poset over a commutative unital ring are described. In \cite{fm25}, authors considered a new family of modified double Poisson brackets and mixed double Poisson algebras. 

In the paper \cite{AKS}, two new types of $\delta$-Poisson and transposed $\delta$-Poisson algebras were studied. The $\delta$-Poisson algebras emerged as a generalization of both Poisson and anti-Poisson algebras. On the other hand, they are closely related to $\delta$-derivations introduced by Filippov in \cite{fil98} (about $\delta$-derivations see \cite{k23,zz}). { It was shown in that paper that transposed $\delta$-Poisson algebras share many known similarities with those studied in \cite{Bai}.} However, unlike transposed Poisson algebras, transposed anti-Poisson algebras possess simple algebras in the complex finite-dimensional case. Furthermore, it was proven that the tensor product of two $\delta$-Poisson (respectively, transposed $\delta$-Poisson) algebras is again a $\delta$-Poisson (respectively, transposed $\delta$-Poisson) algebra.

{Similarly, in the paper \cite{k25} considered the structure and properties of $\delta$-Novikov and $\delta$-Novikov-Poisson algebras, a generalization of Novikov and $\delta$-Novikov-Poisson algebras characterized by a scalar parameter $\delta$.} From the result of this article, we have a crucial difference between Novikov and anti-Novikov algebras. Namely, unlike Novikov algebras, anti-Novikov algebras have complex non-commutative simple finite-dimensional algebras. Moreover, the authors of the article gave some constrictions for obtaining examples of $\delta$-Novikov algebras and proved that the Kantor product of two multiplications of a $\delta$-Novikov-Poisson algebra gives a $\delta$-Novikov algebra. It was shown that the tensor product of two $\delta$-Novikov-Poisson algebras admits a structure of a new $\delta$-Novikov-Poisson algebra under the standard multiplications. Also, they gave relations between $\delta$-derivations, (transposed) $\delta$-Poisson algebras, and $\delta$-Novikov-Poisson algebras. Methods for construction of (transposed) $\delta$-Poisson algebras from a commutative associative algebra with a nonzero $\delta$-derivation are presented. They proved that a $\delta$-Novikov–Poisson algebra under the commutator product gives a transposed $(\delta+1)$-Poisson algebra and they introduced the notion of $\delta$-Gelfand–Dorfman algebras and proved that commutative $\delta$-Gelfand–Dorfman algebras give transposed $(\delta+1)$-Poisson algebras.

Research on transposed Poisson structures has gained significant traction in recent years. In \cite{bfk23}, a comprehensive algebraic and geometric classification of transposed Poisson algebras was provided, further expanding the understanding of their structural properties.  Various works have explored transposed Poisson structures on different classes of Lie algebras. For instance, \cite{kk22, kk23, kkg23} examined such structures on Block Lie algebras and superalgebras, and Witt-type algebras, revealing intricate interactions between transposed Poisson and classical Poisson structures. The study of transposed Poisson structures on low-dimensional Lie algebras was advanced in \cite{ADSS}, where authors investigated the transposed Poisson structures on low dimensional quasi-filiform Lie algebras of maximum length.

Further progress has been made in the classification of transposed Poisson structures on specific algebraic frameworks. In \cite{KK7}, these structures were investigated in the context of upper triangular matrix Lie algebras. Additionally, the application of transposed Poisson structures to incidence algebras was explored in \cite{kkinc}, extending their relevance to combinatorial algebra. A different perspective was provided in \cite{FKL}, where the connection between transposed Poisson algebras and $\frac{1}{2}$-derivations of Lie algebras was analyzed, uncovering new insights into derivation-based algebraic properties.

In \cite{KKhZ, KKhZ1}, transposed Poisson structures were examined in the context of not-finitely graded Witt-type algebras and Virasoro-type algebras, emphasizing their relevance in theoretical physics and representation theory. Using the connection between $\frac{1}{2}$-derivations of Lie algebras and transposed Poisson algebras, descriptions of all transposed Poisson structures on some types of Lie algebras were obtained \cite{AAE, KKh, klv22, lb23, kms, yh21}.

The development of transposed Poisson structures continues to be an active area of research, with recent contributions exploring their connections to Jordan superalgebras \cite{fer23}, and Schr\"{o}dinger algebras \cite{ytk}. Additionally, applications to classification problems in non-associative algebras have been addressed in \cite{k23}, further integrating transposed Poisson structures into the broader landscape of algebraic structures.

This paper provides a complete classification of all transposed $\delta$-Poisson algebra structures on null-filiform associative algebras. All Poisson algebra structures on null-filiform associative algebras were constructed in \cite{AFM}. In \cite{ABT}, authors investigated classifications of all transposed Poisson algebras of the associated associative null-filiform algebra. In Section 2, we introduce the necessary definitions and results that form the basis of our study. Using these foundations, we describe all such structures in Section 3.

\section{PRELIMINARIES}

In this section, we introduce the relevant concepts and known results. Unless stated otherwise, all algebras considered here are over the field $\mathbb{C}$.

We first recall the definition of a Poisson algebra.

\begin{definition}[\cite{Bai}]
Let \(\mathfrak{L}\) be a vector space equipped with two bilinear operations
\[
\cdot, \; [-,-] : \mathfrak{L} \otimes \mathfrak{L} \to \mathfrak{L},
\]
where $(\mathfrak{L}, \cdot)$ is a commutative associative algebra and $(\mathfrak{L}, [-,-])$ is a Lie algebra.

The triple \((\mathfrak{L}, \cdot, [-,-])\) is called a \textbf{$\delta$-Poisson algebra} if:
    \begin{equation}\label{eq:LR}
    [x, y \cdot z] = \delta([x, y] \cdot z + y \cdot [x, z]), \quad \text{for all } x,y,z \in \mathfrak{L}.
    \end{equation}

The triple $(\mathfrak{L},\cdot,[-,-])$ is called a \textbf{transposed $\delta$-Poisson algebra} if:
\begin{equation}
\delta z\cdot [x,y]=[z\cdot x,y]+[x,z\cdot y].\label{eq:dualp}
\end{equation}
\end{definition}
If we take $\delta=1$ and $\delta=2$ respectively in identities (1) and (2), then we obtain the definitions of the Poisson and the transposed Poisson algebras, respectively.

A (transposed) $\delta$-Poisson algebra $\mathfrak{L}$ is called \textit{trivial}, if $\mathfrak{L}\cdot \mathfrak{L}=0$ or $[\mathfrak{L}, \mathfrak{L}]=0$.

Similar to the results in \cite{Bai}, the following result shows that the compatibility relations of the $\delta$-Poisson algebra and the transposed $\delta$-Poisson algebra are independent in the following sense.

\begin{proposition} Let $(\mathfrak{L},\cdot)$ be a commutative associative algebra and $(\mathfrak{L},[-,-])$ be a Lie algebra. Then $(\mathfrak{L},\cdot,[-,-])$ is both a $\delta$-Poisson algebra and a transposed $\delta$-Poisson algebra if and only if
$$x\cdot [y,z]=[x\cdot y,z]=0.$$
\end{proposition}

\begin{proof}  It is easy to see that the statement holds when $\delta = 0$. Therefore, we assume that $\delta \neq 0$. Let $(\mathfrak{L},\cdot,[-,-])$ be a $\delta$-Poisson and transposed $\delta$-Poisson algebra, then according to \cite{AKS}, it satisfies the following identities, respectively:
\begin{equation}\label{eqdPA}
    [x, y \cdot z]+[y, z\cdot x]+ [z, x\cdot y]=0,
    \end{equation}
\begin{equation}\label{eqdTPA}
    x\cdot [y, z]+y\cdot [z, x]+ z\cdot [x,y]=0.
    \end{equation}
for all $x,y,z \in \mathfrak{L}$

It is easy to see, that

    \begin{longtable}{lclcccc}
    $0$ &$=$ &
    $\delta z\cdot[x,y]+ [y,x\cdot z]-[x,y\cdot z] \ \overset{(\hyperlink{eq:LR}{\ref{eq:LR}})}{ =} $\\
    && $\delta\Big(z\cdot[x,y]+[y,x]\cdot z+x\cdot [y,z]-[x,y]\cdot z - y\cdot [x, z]\Big) =$\\
    && $\delta\Big(x\cdot[y,z]+y\cdot[z,x]-z\cdot[x,y]\Big). $ &\\
    \end{longtable}
    
Then by Eq. (\ref{eqdTPA}), we have $z\cdot[x,y]=0.$ By Eq. (\ref{eq:LR}) again, we have $[x,y\cdot z]=0.$
\end{proof}

For an algebra $\mathcal{A}$, we consider the series
\[
\mathcal{A}^1=\mathcal{A}, \quad \ \mathcal{A}^{i+1}=\sum\limits_{k=1}^{i}\mathcal{A}^k \mathcal{A}^{i+1-k}, \quad i\geq 1.
\]

We say that  an  algebra $\mathcal{A}$ is \emph{nilpotent} if $\mathcal{A}^{i}=0$ for some $i \in \mathbb{N}$. The smallest integer $i$ satisfying $\mathcal{A}^{i}=0$ is called the  \emph{index of nilpotency} of $\mathcal{A}$.

\begin{definition}
An $n$-dimensional algebra $\mathcal{A}$ is called null-filiform if $\dim \mathcal{A}^i=(n+ 1)-i,\ 1\leq i\leq n+1$.
\end{definition}

All null-filiform associative algebras were described in  \cite[Proposition 5.3]{MO}.

\begin{theorem}[\cite{MO}] An arbitrary $n$-dimensional null-filiform associative algebra is isomorphic to the algebra:
\[\mu_0^n : \quad e_i \cdot e_j= e_{i+j}, \quad 2\leq i+j\leq n,\]
where $\{e_1, e_2, \dots, e_n\}$ is a basis of the algebra $\mu_0^n$ and the omitted products vanish.
\end{theorem}



\begin{theorem}[\cite{aku}]  \label{3.1} A linear map $\varphi:\mu_0^n\to \mu_0^n $ is an automorphism of the algebra $\mu_0^n$ if and only if the map $\varphi$ has the
following form:
\[
\varphi(e_1)=\sum\limits_{i=1}^nA_ie_i,
\quad
\varphi(e_i)=\sum\limits_{j=i}^n \sum\limits_{k_1 +k_2 +... +k_i=j} A_{k_1} A_{k_2}\dots A_{k_i}e_j, \quad 2\leq i\leq n,
\]
where $A_1\neq0$.
\end{theorem}

In the paper \cite{ABT}, all transposed Poisson algebra structures on null-filiform associative algebras were completely classified.

\begin{theorem} Let $(\mu_0^n, \cdot, [-,-])$ be a transposed Poisson algebra defined on the associative  algebra $(\mu_0^n, \cdot)$. Then the multiplication of $(\mu_0^n, \cdot,[-,-])$ has the following form:

$$\bf{TP}_2(\alpha_2,\dots,  \alpha_n):
\left\{\begin{array}{ll}
e_i\cdot e_j=e_{i+j}, &  2\leq i+j \leq n, \\[1mm]
[e_i,e_j]=(j-i)\sum\limits_{t=i+j-1}^{n}\alpha_{t-i-j+3}e_t, &  3\leq i+j \leq n+1.\\[1mm]
\end{array}\right.$$
\end{theorem}

The following theorem establishes a necessary and sufficient condition for two algebras in the family ${\bf TP}(\alpha_2,\dots, \alpha_n)$ to be isomorphic.

\begin{theorem} Let ${\bf TP}_2(\alpha_2,\dots, \alpha_n)$ and ${\bf TP}'_2(\alpha_2', \dots,   \alpha_n')$ are isomorphic algebras. Then there exists an automorphism $\varphi$ between these algebras such that the following relation holds for $2\leq t\leq n$:
\begin{equation}\label{eq1}
    \sum\limits_{i=2}^t\sum\limits_{k_1  +\cdots+k_i=t} A_{k_1}... A_{k_i}\alpha_i' =\sum\limits_{j=2}^{t}\sum\limits_{i=1}^{t-j+1}  \sum\limits_{k_1+k_2=t-i-j+3} (t-2i-j+3) A_i  A_{k_1} A_{k_2} \alpha_{j}.
\end{equation}
\end{theorem}
\begin{theorem} Let $(\mu_0^n, \cdot, [-,-])$ be a transposed Poisson algebra and $n\geq 5$. Then this algebra is isomorphic to one of the following pairwise non-isomorphic algebras:
\[{\bf TP}_2(1,0,\dots,0), \ {\bf TP}_2(0,\alpha,0, \dots, 0), \ {\bf TP}_2(0,\dots, 0, 1_s,0,\dots,0,\alpha_{2s-3},0, \dots, 0), \ 4\leq s\leq n, \ \alpha\in\mathbb{C}.\]
\end{theorem}
The cases for dimensions 2, 3, and 4 are presented below. In the 2-dimensional associative algebra with multiplication given by $e_1 \cdot e_1 = e_2$, \cite{Bai} proved that every 2-dimensional complex transposed Poisson algebra is isomorphic to one of the following transposed Poisson algebras:
$${\bf TP}_2(0): \ e_1\cdot e_1=e_2; \quad {\bf TP}_2(1): \ e_1\cdot e_1=e_2, \ [e_1,  e_2]=e_2.$$

In the 3-dimensional associative algebra, every 3-dimensional complex transposed Poisson algebra is isomorphic to one of the following transposed Poisson algebras:
\[{\bf TP}_2(1,0), \ {\bf TP}_2(0,\alpha), \ \alpha\in\mathbb{C}.\]

Finally, in the 4-dimensional associative algebra, every 4-dimensional complex transposed Poisson algebra is isomorphic to one of the following transposed Poisson algebras:
\[{\bf TP}_2(1,0,0), \ {\bf TP}_2(0,\alpha,0), \ {\bf TP}_2(0,0,1), \ \alpha\in\mathbb{C}.\]

\section{Transposed $\delta$-Poisson algebra structures on null-filiform associative algebras}

Let $(\mu_0^n, \cdot, [-,-])$ be a transposed $\delta$-Poisson algebra structure defined on the null-filiform associative algebra. To establish the table of multiplications Lie of the transposed $\delta$-Poisson algebra structure, we set
$$[e_1,e_2]=\sum\limits_{t=1}^{n}\alpha_{t}e_t. $$

\begin{theorem}\label{TP} Let $(\mu_0^n, \cdot, [-,-])$ be a transposed $\delta$-Poisson algebra structure defined on the associative  algebra $\mu_0^n$. Then, for $n\geq5$, the following restriction holds:
$$(\delta ^3-3\delta ^2+2\delta)e_3\cdot[e_1,e_2]=0.$$
\end{theorem}

\begin{proof} By considering the identity (\ref{eq:dualp}) for triples  $\{e_1, e_1, e_{i}\}$ and $\{e_{i-1}, e_1, e_{2}\}$:
$$\begin{array}{rcl}
 \delta e_1\cdot[e_1,e_i]&=&[e_2,e_i]+[e_1,e_{i+1}], \\[1mm]
\delta e_{i-1} \cdot [e_1,e_2]&=&[e_i,e_2]+[e_1,e_{i+1}],
\end{array}$$
we derive the following recurrence relation
$$[e_1,e_{i+1}]=\frac{\delta }{2}\Big(e_1\cdot[e_1,e_{i}]+ e_{i-1}\cdot[e_1,e_{2}]\Big).$$

From this, we derive the following products
\begin{equation}\label{1e1e5}[e_1,e_3]=\delta e_1\cdot[e_1,e_2], \ [e_1,e_4]=\frac{\delta^2+\delta}{2} e_2\cdot[e_1,e_2], \  [e_1,e_5]=\frac{\delta ^3+\delta ^2+2\delta }{4}e_3\cdot[e_1,e_2].\end{equation}

Now we consider Eq. (\ref{eq:dualp}) for the triple $\{e_2, e_1, e_3\}:$ 
$$ \delta e_2\cdot[e_1,e_3]=[e_3,e_3]+[e_1,e_{5}].$$
From this we obtain 
\begin{equation}\label{2e1e5}[e_1,e_5]=\delta^2e_3\cdot[e_1,e_2].\end{equation}
From the Eqs. (\ref{1e1e5}) and (\ref{2e1e5}) we get the following
$$(\delta ^3-3\delta ^2+2\delta)e_3\cdot[e_1,e_2]=0.$$

Thus, we get the proof of the theorem. \end{proof}

\begin{corollary}\label{deltavalue} Let $(\mu_0^n, \cdot, [-,-])$ be a transposed $\delta$-Poisson algebra and $n\geq5$. Then

\begin{itemize}
                    \item If $\delta=0$, then $[e_1,e_2]=\sum\limits_{t=1}^{n}\alpha_{t}e_t$ and $[e_1,e_i]=0$ for $3\leq i\leq n;$
                    \item If $\delta=1$, then $[e_1,e_i]=\sum\limits_{t=i-1}^{n}\alpha_{t-i+2}e_t$ for $2\leq i\leq n;$
                    \item If $\delta=2$, then $[e_1,e_i]=(i-1)\sum\limits_{t=i-1}^{n}\alpha_{t-i+2}e_t$ for $2\leq i\leq n;$
                    \item If $\delta ^3-3\delta ^2+2\delta\neq0$, then $\alpha_t=0$ for $1\leq t\leq n-3$ and 
                    $$\begin{array}{lll}
                         [e_1,e_2]=\alpha_{n-2}e_{n-2}+\alpha_{n-1}e_{n-1}+\alpha_ne_n, & [e_1,e_4]=\frac{\delta^2+\delta}{2}\alpha_{n-2}e_n,&   \\[1mm] [e_1,e_3]=\delta (\alpha_{n-2}e_{n-1}+\alpha_{n-1}e_n),& [e_1,e_i]=0, &5\leq i\leq n.\end{array}$$
                  \end{itemize}
\end{corollary}

\begin{theorem}\label{0TPA} Let $(\mu_0^n, \cdot, [-,-])$ be a transposed $0$-Poisson algebra structure defined on the associative  algebra $\mu_0^n$. Then the multiplication of $(\mu_0^n, \cdot,[-,-])$ has the following form:

$${\bf TP}_0(\alpha_1,\dots,  \alpha_n):
\left\{\begin{array}{ll}
e_i\cdot e_j=e_{i+j}, &  2\leq i+j \leq n, \\[1mm]
[e_1,e_2]=\sum\limits_{t=1}^{n}\alpha_{t}e_t.\\[1mm]

\end{array}\right.$$

\end{theorem}
\begin{proof} Let $(\mu_0^n, \cdot, [-,-])$ be a transposed $0$-Poisson algebra. We consider the identity (\ref{eq:dualp}) for the triple $\{e_{i-1}, e_1, e_2\}$:
$$0= [e_{i-1}\cdot e_1,e_2]+[e_1,e_{i-1} \cdot e_2],$$
or
\begin{equation}\label{eqstep11}[e_i,e_2]+[e_1,e_{i+1}]=0.\end{equation}
From Equation (\ref{eqstep11}) and Corollary \ref{deltavalue}, we deduce $[e_i,e_2]=0$ for $2\leq i\leq n.$

Applying induction and the identity (\ref{eq:dualp}) for $3 < i +j$, we establish:
\begin{equation}\label{eq5}
[e_i,e_j]=0.
\end{equation}

We can write
$$0= [e_{i}\cdot e_1,e_j]+[e_1,e_{i} \cdot e_j]=[e_{i+1},e_j]+[e_1,e_{i+j}]= [e_{i+1},e_j].$$

It is known that the Jacobi identity is observed for 3 different elements $\{x,y,z\}$. The identity is satisfied if any 2 of these elements are equal. In our case, the elements $e_3,\dots, e_n $ lie in the center of the Lie algebra. Therefore, the Jacobi identity is satisfied for all elements. Hence, we obtain the transposed $0$-Poisson algebras ${\bf TP}_0(\alpha_1,\dots, \alpha_n)$. 

\end{proof}

The following theorem establishes a necessary and sufficient condition for two algebras in the family ${\bf TP}_0(\alpha_1,\dots, \alpha_n)$ to be isomorphic.

\begin{theorem}\label{0lemm} Let ${\bf TP}_0(\alpha_1,\dots, \alpha_n)$ and ${\bf TP}_0'(\alpha_1', \dots,   \alpha_n')$ are isomorphic algebras. Then there exists an automorphism $\varphi$ between these algebras such that the following relation holds for $1\leq t\leq n$:

\begin{equation}\label{0eq1}
    \sum\limits_{i=1}^t\sum\limits_{k_1  +\cdots+k_i=t} A_{k_1}... A_{k_i}\alpha_i' = A_1^3 \alpha_{t} .
\end{equation}
\end{theorem}
\begin{proof}
Using the automorphism of the algebra $\mu_0^n$ from Theorem \ref{3.1}, we introduce the following notations:

\[
e_1'=\varphi(e_1), \quad
e_i'=\varphi(e_i),\quad 2\leq i\leq n.
\]
Thus, we consider
 $$\begin{array}{lcl}
    [e_1',e_2'] &= &\sum\limits_{i=1}^{n}\alpha_{i}'e_i'=\sum\limits_{i=1}^n \alpha_i'\sum\limits_{j=i}^n\sum\limits_{k_1 +\dots+k_i=j} A_{k_1}...A_{k_i}e_j \\[1mm]
    &=& \sum\limits_{i=1}^n \sum\limits_{j=i}^n \sum\limits_{k_1+\dots +k_i=j}\alpha_i' A_{k_1} ...A_{k_i}e_j=\sum\limits_{t=1}^n \sum\limits_{i=1}^t\sum\limits_{k_1  +\cdots+k_i=t} \alpha_i' A_{k_1} ...A_{k_i}e_t.  \\
        \end{array}$$

On the other hand, we have
 $$ [e_1',e_2']=[\sum\limits_{i=1}^n A_i e_i, \sum\limits_{j=2}^n\sum\limits_{k_1+k_2=j} A_{k_1} A_{k_2} e_j]=\sum\limits_{k_1+k_2=2} A_1 A_{k_1} A_{k_2} \sum\limits_{t=1}^n \alpha_{t} e_t = \sum\limits_{t=1}^n A_1^3 \alpha_{t} e_t.$$

Comparing the coefficients of the obtained expressions for the basis elements for $1\leq t\leq n$, we get the following restrictions:
$$\sum\limits_{i=1}^t\sum\limits_{k_1  +\cdots+k_i=t} A_{k_1}... A_{k_i}\alpha_i' = A_1^3 \alpha_{t} .$$

\end{proof}

\begin{lemma} Let ${\bf TP}(0,\dots,0, \alpha_s, \dots,\alpha_n),$ with $\alpha_s\neq0,$  be a transposed Poisson algebra defined above. Then, there exists $A\in\mathbb{C}$ such that the relation $\alpha_{s}'=\alpha_{s}A^{3-s}$ holds for any $1\leq s\leq n.$
\end{lemma}

\begin{proof} Let ${\bf TP}_0(0,\dots,0,\alpha_s,\dots, \alpha_n)$ be a transposed $0$-Poisson algebra and consider a general change of basis. Then, for some natural number $s$, we have the following restriction (\ref{0eq1}):
$$ \sum\limits_{i=1}^s\sum\limits_{k_1  +\cdots+k_i=s} A_{k_1}... A_{k_i}\alpha_i' = A_1^3 \alpha_{s} .$$

Now, considering that $\alpha_i=0,$ $\alpha'_i=0,$ for $1\leq i< s-1$, this reduces to
$$ \sum\limits_{k_1  +\cdots+k_i=s} A_{k_1}... A_{k_i}\alpha_s' = A_1^3 \alpha_{s}$$
which simplifies further to $A_1^s\alpha_s'= A_1^3  \alpha_{s}.$ From this equality, we obtain the relation $\alpha_{s}'=\alpha_{s}A_1^{3-s}.$

\end{proof}

\begin{lemma}\label{0lem2}  If $\alpha_1\neq0,$ then ${\bf TP}_0(\alpha_1, \dots, \alpha_n)$ is
isomorphic to the algebra ${\bf TP}_0(1, 0, \dots, 0)$.
\end{lemma}

\begin{proof} Let $\alpha_1\neq0$ and consider the following change of basis:
$e_i'=\frac{1}{\sqrt{\alpha_1^{i}}}e_i, 1\leq i\leq n,$ then we have $\alpha_1'=1.$

Next, consider another change of basis $e_1'=e_1+\alpha_2e_2.$  From the relations $e_i\cdot e_j=e_{i+j}$ for $2\leq i+j\leq n,$ we obtain
$$e_i'=\sum\limits_{t=0}^{i}\left(\begin{array}{cc}
     i  \\
     t
\end{array}\right)\alpha_2^te_{i+t}, \ \ 1\leq i\leq n,$$
We conclude that $\alpha_2'=0.$

Now, we prove by induction that it is possible to set $\alpha_j=0,$ for all $2\leq j\leq n$. The base case $j = 2$ is already established. Assuming the claim holds for some $j$, we show it also holds for $j+1$.  Consider the change of basis: $$e_1'=e_1+\alpha_{j+1}e_{j+1}.$$  Using $e_i\cdot e_j=e_{i+j}$ for $2\leq i+j\leq n,$ we derive
$$e_i'=\sum\limits_{t=0}^{i}\left(\begin{array}{cc}
     i  \\
     t
\end{array}\right)\alpha_{j+1}^te_{i+tj}, \ \ 1\leq i\leq n,$$
where we assume that $e_t=0$ for $t>n$ in this sum. We obtain that $\alpha_{j+1}=0.$

By induction, we have $\alpha_j=0,$ for all $2\leq j\leq n$, proving that the algebra $\mathfrak{L}$ is isomorphic to ${\bf TP}_0(1, 0, \dots, 0)$, with the following multimplication rules:
$$\left\{\begin{array}{ll}
e_i\cdot e_j=e_{i+j}, &  2\leq i+j \leq n, \\[1mm]
[e_1,e_2]=e_{1}.\\[1mm]
\end{array}\right.$$
\end{proof}

\begin{lemma}\label{0lem3} If $ \alpha_2\neq0,$ then ${\bf TP}_0(0, \alpha_2, \dots, \alpha_n)$ is
isomorphic to the algebra ${\bf TP}_0(0,1, 0, \dots, 0)$.
\end{lemma}

\begin{proof} Let $\alpha_2\neq0$ and consider the following change of basis:
$e_i'=\frac{1}{{\alpha_2^{i}}}e_i, 1\leq i\leq n$. Then, we have $\alpha_2'=1.$

Next, consider the following change of basis: $$e_1'=e_1+\frac{\alpha_3}{2}e_2.$$
From the products $e_i\cdot e_j=e_{i+j}$ for $2\leq i+j\leq n,$ we obtain
$$e_i'=\sum\limits_{t=0}^{i}\left(\begin{array}{cc}
     i  \\
     t
\end{array}\right)\left(\frac{\alpha_3}{2}\right)^te_{i+t}, \ \ 1\leq i\leq n,$$
where we assume that $e_t=0$ for $t>n$.
We conclude that $\alpha_3'=0.$

Now, we prove by induction that we can eliminate $\alpha_j$ for all $3\leq j\leq n$. The base case $j = 3$ follows from the above step. Assuming that $\alpha_j=0$ holds for some $j\geq 3$, we prove it for $j+1$. Consider the change of basis: $$e_1'=e_1+\frac{\alpha_{j+1}}{2}e_{j}.$$  From the products $e_i\cdot e_j=e_{i+j}$ for $2\leq i+j\leq n,$ we derive
$$e_i'=\sum\limits_{t=0}^{i}\left(\begin{array}{cc}
     i  \\
     t
\end{array}\right)\left(\frac{\alpha_{j+1}}{2}\right)^te_{i+t(j-1)}, \ \ 1\leq i\leq n,$$
where we again assume that $e_t=0$ for $t>n$. We have that $\alpha_{j+1}=0.$

Thus, we have shown that $\alpha_j=0,$ for $3\leq j\leq n$, and the algebra $\mathfrak{L}$ is isomorphic to the algebra ${\bf TP}_0(0, 1, 0, \dots, 0)$:
$$\left\{\begin{array}{ll}
e_i\cdot e_j=e_{i+j}, &  2\leq i+j \leq n, \\[1mm]
[e_1,e_2]=e_{2}.\\[1mm]
\end{array}\right.$$

\end{proof}

\begin{lemma}\label{0lem4} Let $\mathfrak{L}$ be the algebra ${\bf TP}_0(0,0,\alpha_3 \dots, \alpha_n)$. If $\alpha_3\neq0,$ then it is
isomorphic to the algebra ${\bf TP}_0(0,0, \alpha,0,0, \dots, 0)$.
\end{lemma}

\begin{proof} From the equality \eqref{0eq1}, we obtain the following equality in general substitution.
$ \alpha_3'=\alpha_3.$

Now, consider the change of basis: $e_1'=e_1+\frac{\alpha_4}{3\alpha_3}e_2.$  From the products $e_i\cdot e_j=e_{i+j}$ for $2\leq i+j\leq n,$ we derive
$$e_i'=\sum\limits_{t=0}^{i}\left(\begin{array}{cc}
     i  \\
     t
\end{array}\right)\left(\frac{\alpha_4}{3\alpha_3}\right)^te_{i+t}, \ \ 1\leq i\leq n,$$
where we assume that $e_t=0$ for $t>n$ in this sum. We obtain $\alpha_4'=0.$

Now we prove by induction that it is possible to set $\alpha_j=0,$ for $4\leq j\leq n$. The base case $j = 4$ holds by the above argument. Now, assuming it holds for some $j$, we prove it for $j + 1$. Consider basis change: $e_1'=e_1+\frac{\alpha_{j+1}}{3\alpha_3}e_{j-1}.$  From the products $e_i\cdot e_j=e_{i+j}$ for $2\leq i+j\leq n,$ we derive
$$e_i'=\sum\limits_{t=0}^{i}\left(\begin{array}{cc}
     i  \\
     t
\end{array}\right)\left(\frac{\alpha_{j+1}}{3\alpha_3}\right)^te_{i+t(j-2)}, \ \ 1\leq i\leq n,$$
where we assume that $e_t=0$ for $t>n$ in this sum. We obtain that $\alpha_{j+1}=0.$

Thus, we have shown that $\alpha_j=0,$ for $4\leq j\leq n$, and the algebra $\mathfrak{L}$ is isomorphic to the algebra ${\bf TP}_0(0,0, \alpha,0, \dots, 0)$:
$$\left\{\begin{array}{ll}
e_i\cdot e_j=e_{i+j}, &  2\leq i+j \leq n, \\[1mm]
[e_1,e_2]=\alpha e_{3}
\end{array}\right.$$

\end{proof}

\begin{lemma}\label{0lems} Let $\mathfrak{L}$ be the algebra ${\bf TP}_0(0, \dots,0,\alpha_s,\dots, \alpha_n)$. If $ \alpha_s\neq0, \ s\geq 4$ then it is
isomorphic to the algebra ${\bf TP}_0(0,\dots, 0, 1_s,0,\dots,0)$.
\end{lemma}

\begin{proof}
Let $\alpha_s\neq0$ and consider the change of basis as follows
$ e_i'=\alpha_s^{\frac{i}{s-3}}e_i, \ \ 1\leq i\leq n.$ Then, we have $\alpha_s'=1.$

Now consider the change of basis as follows: $e_1'=e_1+\frac{\alpha_{s+1}}{s}e_2.$  From the products $e_i\cdot e_j=e_{i+j}$ for $2\leq i+j\leq n,$ we derive
$$e_i'=\sum\limits_{t=0}^{i}\left(\begin{array}{cc}
     i  \\
     t
\end{array}\right)\left(\frac{\alpha_{s+1}}{s}\right)^te_{i+t}, \ \ 1\leq i\leq n,$$
where we assume that $e_t=0$ for $t>n$ in this sum.
We can derive that $\alpha_{s+1}=0.$

Now we prove by induction using successive changes of the basis elements that it is possible to make $\alpha_j=0,$ for $s+1\leq j\leq n$. If $j = i$, the relationship is fulfilled according to the above equalities. Now, that it holds for some $j$, we prove it for $j + 1$. We consider the change of the basis element $e_1$ as $e_1'=e_1+\frac{\alpha_{j+1}}{s}e_{j-s+2}.$  From the products  $e_i\cdot e_j=e_{i+j}$ for $2\leq i+j\leq n,$ we derive
$$e_i'=\sum\limits_{t=0}^{i}\left(\begin{array}{cc}
     i  \\
     t
\end{array}\right)\left(\frac{\alpha_{j+1}}{s}\right)^te_{i+t(j-s+1)}, \ \ 1\leq i\leq n,$$
where we assume that $e_t=0$ for $t>n$ in this sum. We can derive that $\alpha_{j+1}'=0.$

Thus, we have shown that $\alpha_j=0,$ for $s+1\leq j\leq n$, and the algebra $\mathfrak{L}$ is isomorphic to the algebra ${\bf TP}_0(0,\dots, 0, 1_s,0,\dots,0)$:
$$\left\{\begin{array}{ll}
e_i\cdot e_j=e_{i+j}, &  2\leq i+j \leq n, \\[1mm]
[e_1,e_2]=e_{s}.\\[1mm]
\end{array}\right.$$

\end{proof}

\begin{theorem}\label{tdelta=0} Let $(\mu_0^n, \cdot, [-,-])$ be a transposed $0$-Poisson algebra and $n\geq 5$. Then, this algebra is isomorphic to one of the following pairwise non-isomorphic algebras:
\[{\bf TP}_0(1,0,\dots,0), \  {\bf TP}_0(0,1,0,\dots,0), \ {\bf TP}(0,0,\alpha,0, \dots, 0), \ {\bf TP}_0(0,\dots, 0, 1_s,0,\dots,0), \ 4\leq s\leq n, \ \alpha\in\mathbb{C}.\]
\end{theorem}

\begin{proof} Let $(\mu_0^n, \cdot, [-,-])$ be a transposed Poisson algebra structure defined on the associative  algebra $\mu_0^n$. Then Theorem \ref{0TPA} implies that it is isomorphic to the algebra ${\bf TP}_0(\alpha_1,\dots, \alpha_n)$. Moreover,
\begin{itemize}
    \item If $\alpha_1\neq 0$ then from Lemma \ref{0lem2} we have ${\bf TP}_0(1,0,\dots,0)$ algebra;
    \item If $\alpha_1=0, \ \alpha_2\neq 0$ then from Lemma \ref{0lem3} we have ${\bf TP}_0(0,1,0, \dots, 0)$ algebra, where $\alpha\neq0$;
    \item If $\alpha_1=\alpha_2=0, \ \alpha_3\neq 0$ then from Lemma \ref{0lem4} we have ${\bf TP}_0(0,0,\alpha,0, \dots, 0)$ algebra;
    \item If $\alpha_1=\dots=\alpha_{s-1}=0, \ \alpha_s\neq 0, \ s\geq4$ then from Lemma \ref{0lems} we have the algebra $${\bf TP}_0(0,\dots, 0, 1_s,0,\dots,0), $$
    \item If $\alpha_i=0, \ 1\leq i\leq n$ then we have ${\bf TP}_0(0,\dots,0)$ algebra.
\end{itemize} \end{proof}

Now we consider the transposed $1$-Poisson algebra structures on null-filiform associative algebras.

\begin{theorem}\label{1TPA} Let $(\mu_0^n, \cdot, [-,-])$ be a transposed $1$-Poisson algebra structure defined on the associative  algebra $\mu_0^n$. Then, the multiplication of $(\mu_0^n, \cdot,[-,-])$ has the following form:

$${\bf TP}_1(\alpha_2,\dots,  \alpha_n):
\left\{\begin{array}{ll}
e_i\cdot e_j=e_{i+j}, &  2\leq i+j \leq n, \\[1mm][e_1,e_i]=\sum\limits_{t=i}^{n}\alpha_{t-i+2}e_t, & 2\leq i\leq n.\\[1mm]
\end{array}\right.$$

\end{theorem}
\begin{proof} Let $(\mu_0^n, \cdot, [-,-])$ be a transposed $1$-Poisson algebra. We consider the identity (\ref{eq:dualp}) for the triple $\{e_{i-1}, e_1, e_2\}, \ 2\leq i\leq n+1$:
$$e_{i-1} \cdot [e_1,e_{2}]= [e_{i-1}\cdot e_1,e_2]+[e_1,e_{i-1} \cdot e_2],$$
or
\begin{equation}\label{eqstep2}[e_i,e_2]+[e_1,e_{i+1}]=\sum\limits_{t=i}^{n}\alpha_{t-i+1}e_t.\end{equation}
From Corollary (\ref{deltavalue}), we deduce
$[e_i,e_2]=0$ for $2\leq i\leq n.$

Applying induction and the identity (\ref{eq:dualp}) for $ 2\leq i,j\leq n$ we establish:
\begin{equation}\label{1eq5}
[e_i,e_j]=0.\quad \quad\end{equation}

We will prove this equality by induction on the value of $j$. For $j=2$, the expression $[e_i,e_2]$,  is valid for $\ 2\leq i\leq n$. Assuming that $[e_i,e_j]=0$, holds for $2\leq i\leq n$, we will prove the equality for $j+1$. We can write
$$0=e_{1} \cdot [e_{i},e_{j}]= [e_{1}\cdot e_{i},e_j]+[e_i,e_{1} \cdot e_j]=[e_{i+1},e_j]+[e_i,e_{j+1}]=[e_{i},e_{j+1}].$$

Thus, $[e_{i},e_j]=0$, for $\ 2\leq j\leq n.$  We have, therefore, proven the equality \eqref{1eq5}.

Next, applying the Jacobi identity to $\{e_1, e_2, e_n\}$, we get:
$$0=[[e_1,e_2],e_{n}]+[[e_2,e_{n}],e_1]+[[e_{n},e_1],e_2]=$$
$$=[\sum\limits_{t=1}^{n}\alpha_{t}e_t,e_{n}]-[\sum\limits_{t=n-1}^{n}\alpha_{t-n+2}e_t,e_2]=\alpha_1(\alpha_1e_{n-1}+\alpha_2e_n).$$
From the last equality, we obtain the relation $\alpha_1=0$.

Hence, we obtain the transposed $1$-Poisson algebras ${\bf TP}_1(\alpha_2,\dots, \alpha_n)$ given by the following multiplications:

$$\left\{\begin{array}{ll} e_i\cdot e_j = e_{i+j} , & 2 \leq  i + j \leq n, \\[1mm]
[e_1,e_i]=\sum\limits_{t=i}^{n}\alpha_{t-i+2}e_t,&2\leq i\leq n.\\[1mm]
\end{array}\right.$$
\end{proof}

The following theorem establishes a necessary and sufficient condition for two algebras in the family ${\bf TP}_1(\alpha_2,\dots, \alpha_n)$ to be isomorphic.

\begin{theorem} Let ${\bf TP}_1(\alpha_2,\dots, \alpha_n)$ and ${\bf TP}'_1(\alpha_2', \dots,   \alpha_n')$ are isomorphic algebras. Then there exists an automorphism $\varphi$ between these algebras such that the following relation holds for $2\leq t\leq n$:

\begin{equation}\label{1eq1}
    \sum\limits_{i=2}^t\sum\limits_{k_1  +\cdots+k_i=t} A_{k_1}... A_{k_i}\alpha_i' =\sum\limits_{j=2}^{t} \sum\limits_{k_1+k_2=t-j+2}A_1 A_{k_1} A_{k_2} \alpha_{j}.
\end{equation}
\end{theorem}

\begin{proof}
Using the automorphism of the algebra $\mu_0^n$ from Theorem \ref{3.1}, we introduce the following notations:

\[
e_1'=\varphi(e_1), \quad
e_i'=\varphi(e_i),\quad 2\leq i\leq n.
\]

Thus, we consider
 $$\begin{array}{lcl}
    [e_1',e_2'] &= &\sum\limits_{i=2}^{n}\alpha_{i}'e_i'=\sum\limits_{i=2}^n \alpha_i'\sum\limits_{j=i}^n\sum\limits_{k_1 +... +k_i=j} A_{k_1}...A_{k_i}e_j \\[1mm]
    &=& \sum\limits_{i=2}^n \sum\limits_{j=i}^n \sum\limits_{k_1+\dots +k_i=j}\alpha_i' A_{k_1} ...A_{k_i}e_j=\sum\limits_{t=2}^n \sum\limits_{i=2}^t\sum\limits_{k_1  +\cdots+k_i=t} \alpha_i' A_{k_1} ...A_{k_i}e_t.  \\
        \end{array}$$

On the other hand, we have
 $$\begin{array}{lcl}
 [e_1',e_2']&=&[\sum\limits_{i=1}^n A_i e_i, \sum\limits_{j=2}^n\sum\limits_{k_1+k_2=j} A_{k_1} A_{k_2} e_j]=[A_{1} e_1,\sum\limits_{j=2}^n  \sum\limits_{k_1+k_2=j}A_{k_1}A_{k_2}e_j]\\[1mm]
 &=&\sum\limits_{j=2}^n \sum\limits_{k_1+k_2=j} A_1 A_{k_1} A_{k_2}[e_1,e_j] \\[1mm]
&=&\sum\limits_{j=2}^n \sum\limits_{k_1+k_2=j} A_1 A_{k_1} A_{k_2}\sum\limits_{t=j}^n \alpha_{t-j+2} e_t\\[1mm]
&=&\sum\limits_{t=2}^n \sum\limits_{j=2}^{t}\sum\limits_{k_1+k_2=j} A_1 A_{k_1} A_{k_2} \alpha_{t-j+2} e_t.\\[1mm]
\end{array}$$

Comparing the coefficients of the obtained expressions for the basis elements for $2\leq t\leq n$, we get the following restrictions:
$$\sum\limits_{i=2}^t\sum\limits_{k_1  +\cdots+k_i=t} A_{k_1}... A_{k_i}\alpha_i'=\sum\limits_{j=2}^{t} \sum\limits_{k_1+k_2=j}A_1 A_{k_1} A_{k_2} \alpha_{t-j+2}.$$

If we write the values of the parameter $\alpha$ on the right-hand side of the equation in increasing order, we obtain the following equality.

$$\sum\limits_{i=2}^t\sum\limits_{k_1  +\cdots+k_i=t} A_{k_1}... A_{k_i}\alpha_i'=\sum\limits_{j=2}^{t} \sum\limits_{k_1+k_2=t-j+2}A_1 A_{k_1} A_{k_2} \alpha_{j}.$$
\end{proof}

\begin{lemma}\label{1lem9} Let ${\bf TP}_1(0,\dots,0, \alpha_s, \dots,\alpha_n),$ with $\alpha_s\neq0,$  be a transposed $1$-Poisson algebra defined above. Then, there exists $A\in\mathbb{C}$ such that the relation $\alpha_{s}'=\alpha_{s}A^{3-s}$ holds for any $2\leq s\leq n.$
\end{lemma}

\begin{proof} Let ${\bf TP}_1(\alpha_2,\dots, \alpha_n)$ be a transposed $1$-Poisson algebra structure on $\mu_0^n$, and consider a general change of basis.  Let $\alpha_i=0$ for $2\leq i< s-1$ then $\alpha_i'=0$. Then, for some natural number $s$, we have the following restriction (\ref{1eq1}):
$$ \sum\limits_{i=2}^s\sum\limits_{k_1  +\cdots+k_i=s} A_{k_1}... A_{k_i}\alpha_i' =\sum\limits_{j=2}^{s}\sum\limits_{k_1+k_2=j}  A_1  A_{k_1} A_{k_2} \alpha_{s-j+2}$$
or
$$ \sum\limits_{k_1  +k_2=2}A_{k_1}A_{k_2}\alpha_{2}'+\sum\limits_{k_1  +\cdots+k_s=s} A_{k_1}... A_{k_s}\alpha_s' =\sum\limits_{k_1  +k_2=2}A_{1}A_{k_1}A_{k_2}\alpha_{2}+\sum\limits_{k_1+k_2=s}  A_1  A_{k_1} A_{k_2} \alpha_{s}.$$

For $s=2$, we derive $\alpha_2'=A_1\alpha_2$. In the next steps we will have the following 
$$ \sum\limits_{k_1  +\dots+k_s=s} A_{k_1} ... A_{k_s}\alpha_s'= \sum\limits_{k_1+k_2=2}A_1 A_{k_1} A_{k_2} \alpha_{s}.$$
From this equality, we obtain the relation $\alpha_s'=\alpha_{s}A_1^{3-s}.$
\end{proof}

\begin{lemma}\label{NL2} Let $\mathfrak{L}$ be the algebra ${\bf TP}_1(\alpha_2,0,\dots, 0, \alpha_s,\dots, \alpha_n)$ and $3\leq s \leq n+1$. If $\alpha_2\neq0, \  \alpha_s\neq0$ then it is isomorphic to the algebra ${\bf TP}_1(1,0, \dots, 0, \alpha_s,0, \dots, 0)$.
\end{lemma}

\begin{proof} By general substitution and according to Lemma \ref{1lem9}, we obtain the following equality.

$$\alpha_2'=A_1\alpha_2, \quad \alpha_s'=\frac{\alpha_s}{A_1^{s-3}}.$$
According to the theorem's condition,  $\alpha_2\neq 0$. If we perform the substitution $e_i'=\frac{1}{\alpha_2^i}e_i$, for $ \ 1\leq i\leq n$, then we obtain the relation $\alpha_2'=1$.

Next, consider another change of basis for $e_1:$ $$e_1'=e_1+\frac{\alpha_{s+1}}{(s-2)\alpha_s}e_2.$$  From the relations $e_i\cdot e_j=e_{i+j}$ for $2\leq i+j\leq n,$ we obtain
$$e_i'=\sum\limits_{t=0}^{i}\left(\begin{array}{cc}
     i  \\
     t
\end{array}\right)\left(\frac{\alpha_{s+1}}{(s-2)\alpha_s}\right)^te_{i+t}, \ \ 1\leq i\leq n,$$
where we assume that $e_t=0$ for $t>n$.
We conclude that $\alpha_{s+1}=0.$

Now, we prove by induction that it is possible to set $\alpha_j=0,$ for all $s+1\leq j\leq n$. The base case $j = s+1$ is already established. Assuming that the claim holds for some $j$, we show that it also holds for $j + 1$.  Consider the change of basis: $$e_1'=e_1+\frac{\alpha_{j+1}}{(s-2)\alpha_s}e_{j-s+2}.$$  Using $e_i\cdot e_j=e_{i+j}$ for $2\leq i+j\leq n,$ we derive
$$e_i'=\sum\limits_{t=0}^{i}\left(\begin{array}{cc}
     i  \\
     t
\end{array}\right)\left(\frac{\alpha_{j+1}}{(s-2)\alpha_s}\right)^te_{i+t(j-s+1)}, \ \ 1\leq i\leq n,$$
where we assume that $e_t=0$ for $t>n$ in this sum. We obtain that $\alpha_{j+1}=0.$

By induction, we have $\alpha_j=0,$ for all $s+1\leq j\leq n$, proving that the algebra $\mathfrak{L}$ is isomorphic to ${\bf TP}_1(1,0, \dots, 0, \alpha_s,0, \dots, 0)$, with the following multimplication rules:
$$\left\{\begin{array}{ll}
e_i\cdot e_j=e_{i+j}, &  2\leq i+j \leq n, \\[1mm]
[e_1,e_i]=e_i+\alpha e_{s+i-2}, &  2\leq i \leq n-s+2.\\[1mm]
\end{array}\right.$$ \end{proof}

\begin{lemma}\label{NL3} Let $\mathfrak{L}$ be the algebra ${\bf TP}_1(0,\alpha_3,\dots, \alpha_n)$. If $\alpha_3\neq0$ then it is isomorphic to the algebra ${\bf TP}_1(0,\alpha,0 , \dots, 0)$.
\end{lemma}

\begin{proof}  By general substitution and according to Lemma \ref{1lem9}, we obtain $\alpha_3'=\alpha_3.$ Next, consider another change of basis $e_1'=e_1+\frac{\alpha_4}{\alpha_3}e_2.$  From the relations $e_i\cdot e_j=e_{i+j}$ for $2\leq i+j\leq n,$ we obtain
$$e_i'=\sum\limits_{t=0}^{i}\left(\begin{array}{cc}
     i  \\
     t
\end{array}\right)\left(\frac{\alpha_4}{\alpha_3}\right)^te_{i+t}, \ \ 1\leq i\leq n,$$
where we assume that $e_t=0$ for $t>n$.
We conclude that $\alpha_4'=0.$

Now, we prove by induction that it is possible to set $\alpha_j=0,$ for all $4\leq j\leq n$. The base case $j = 4$ is already established. Assuming that the claim holds for some $j$, we show that it also holds for $j + 1$.  Consider the change of basis: $$e_1'=e_1+\frac{\alpha_{j+1}}{\alpha_3}e_{j-1}.$$  Using $e_i\cdot e_j=e_{i+j}$ for $2\leq i+j\leq n,$ we derive
$$e_i'=\sum\limits_{t=0}^{i}\left(\begin{array}{cc}
     i  \\
     t
\end{array}\right)\left(\frac{\alpha_{j+1}}{\alpha_3}\right)^te_{i+t(j-2)}, \ \ 1\leq i\leq n,$$
where we assume that $e_t=0$ for $t>n$ in this sum. We have that $\alpha_{j+1}=0.$

By induction, we have $\alpha_j=0,$ for all $4\leq j\leq n$, proving that the algebra $\mathfrak{L}$ is isomorphic to ${\bf TP}_1(0, \alpha,0, \dots, 0)$, with the following multimplication rules:
$$\left\{\begin{array}{ll}
e_i\cdot e_j=e_{i+j}, &  2\leq i+j \leq n, \\[1mm]
[e_1,e_i]=\alpha e_{i+1}, &  2\leq i \leq n-1.\\[1mm]
\end{array}\right.$$
\end{proof}

\begin{lemma}\label{NL4} Let $\mathfrak{L}$ be the algebra ${\bf TP}_1(0,\dots,0,\alpha_s,\dots, \alpha_n)$ and $s\geq 4$. If $\alpha_s\neq0$, then it is
isomorphic to the algebra ${\bf TP}_1(0,0, \dots, 0, 1_s,0, \dots, 0)$.
\end{lemma}

\begin{proof} By general substitution and according to Lemma \ref{1lem9}, we obtain the following equality

$$\alpha_s'=\frac{\alpha_s}{A_1^{s-3}}.$$
If we perform the substitution $e_i'=\sqrt[s-3]{\alpha_s^{i}}e_i$, for $ \ 1\leq i\leq n$, then we obtain the relation $\alpha_s'=1$.

Next, consider another change of basis for $e_1:$ $$e_1'=e_1+\frac{\alpha_{s+1}}{(s-2)}e_2.$$  From the relations $e_i\cdot e_j=e_{i+j}$ for $2\leq i+j\leq n,$ we obtain
$$e_i'=\sum\limits_{t=0}^{i}\left(\begin{array}{cc}
     i  \\
     t
\end{array}\right)\left(\frac{\alpha_{s+1}}{(s-2)}\right)^te_{i+t}, \ \ 1\leq i\leq n,$$
where we assume that $e_t=0$ for $t>n$.
We conclude that $\alpha_{s+1}=0.$

Now, we prove by induction that it is possible to set $\alpha_j=0,$ for all $s+1\leq j\leq n$. The base case $j = s+1$ is already established. Assuming that the claim holds for some $j$, we show that it also holds for $j + 1$.  Consider the change of basis: $$e_1'=e_1+\frac{\alpha_{j+1}}{(s-2)}e_{j-s+2}.$$  Using $e_i\cdot e_j=e_{i+j}$ for $2\leq i+j\leq n,$ we derive
$$e_i'=\sum\limits_{t=0}^{i}\left(\begin{array}{cc}
     i  \\
     t
\end{array}\right)\left(\frac{\alpha_{j+1}}{(s-2)}\right)^te_{i+t(j-s+1)}, \ \ 1\leq i\leq n,$$
where we assume that $e_t=0$ for $t>n$ in this sum. We conclude that $\alpha_{j+1}=0.$

By induction, we have $\alpha_j=0,$ for all $s+1\leq j\leq n$, proving that the algebra $\mathfrak{L}$ is isomorphic to ${\bf TP}_1(0,0, \dots, 0, 1_s,0, \dots, 0)$, with the following multimplication rules:
$$\left\{\begin{array}{ll}
e_i\cdot e_j=e_{i+j}, &  2\leq i+j \leq n, \\[1mm]
[e_1,e_i]=e_{s+i-2}, &  2\leq i \leq n+2-s.\\[1mm]
\end{array}\right.$$
\end{proof}

\begin{theorem} Let $(\mu_0^n, \cdot, [-,-])$ be a $1$-transposed Poisson algebra structure defined on the associative  algebra $\mu_0^n$. Then this algebra is isomorphic to one of the following pairwise non-isomorphic algebras:
\[{\bf TP}_1(1,0, \dots, 0, \alpha_s,0, \dots, 0), \ {\bf TP}_1(0,\alpha,0 , \dots, 0), \
{\bf TP}_1(0, \dots, 0, 1_p,0, \dots, 0),
 \ 3\leq s\leq n, \ 4\leq p\leq n, \ \alpha\in\mathbb{C}.\]
\end{theorem}

\begin{proof} Let $(\mu_0^n, \cdot, [-,-])$ be a transposed $1$-Poisson algebra structure defined on the associative  algebra $\mu_0^n$. Then Theorem \ref{1TPA} implies that it is isomorphic to the algebra ${\bf TP}_1(\alpha_2,\dots, \alpha_n)$. Moreover,
\begin{itemize}
     \item If $\alpha_2\neq0, \ \alpha_3=\dots=\alpha_{s-1}=0$ and $\alpha_s\neq 0$ for $s\geq 3$, then from Lemma \ref{NL2} we have ${\bf TP}_1(1,0, \dots, 0, \alpha_s,0, \dots, 0)$ algebra;
    \item If $\alpha_2=0, \ \alpha_3\neq 0$ then from Lemma \ref{NL3} we have ${\bf TP}_1(0,\alpha,0 , \dots, 0)$ algebra;
    \item If $\alpha_2=\dots=\alpha_{p-1}=0, \ \alpha_p\neq 0, \ p\geq4$ then from Lemma \ref{NL4} we have the algebra
    $${\bf TP}_1(0, \dots, 0, 1_p,0, \dots, 0), \ 4\leq p\leq n;$$
    \item If $\alpha_i=0, \ 2\leq i\leq n$ then we have ${\bf TP}_1(0,\dots,0)$ algebra.
\end{itemize} \end{proof}

Now we consider transposed $\delta$-Poisson algebra structures on null-filiform associative algebras for $\delta\neq 0,1,2$.

\begin{theorem}\label{TP} Let $(\mu_0^n, \cdot, [-,-])$ be a transposed $\delta$-Poisson algebra structure defined on the associative  algebra $\mu_0^n, \ n\geq5$ and $\delta\neq0,1,2$. Then the multiplication of $(\mu_0^n, \cdot,[-,-])$ has the following form:

$${\bf TP}_{\delta}(\alpha_{n-2},\alpha_{n-1},  \alpha_n):
\left\{\begin{array}{ll} e_i\cdot e_j=e_{i+j},   &  \quad      2\leq i+j \leq n, \\[1mm]
[e_1,e_2]=\alpha_{n-2}e_{n-2}+\alpha_{n-1}e_{n-1}+\alpha_ne_n, \\[1mm]
[e_1,e_3]=\delta (\alpha_{n-2}e_{n-1}+\alpha_{n-1}e_n),    & \\[1mm] [e_2,e_3]=\frac{\delta^2-\delta}{2}\alpha_{n-2}e_n, &    \\[1mm] [e_1,e_4]=\frac{\delta^2+\delta}{2}\alpha_{n-2}e_n. & \\[1mm]
\end{array}\right.$$

\end{theorem}

\begin{proof} Let $(\mu_0^n, \cdot, [-,-])$ be a transposed $\delta$-Poisson algebra. Then, according to Corollary \ref{deltavalue}, we have the multiplications
 $$\begin{array}{lll}
                         [e_1,e_2]=\alpha_{n-2}e_{n-2}+\alpha_{n-1}e_{n-1}+\alpha_ne_n, & [e_1,e_4]=\frac{\delta^2+\delta}{2}\alpha_{n-2}e_n,&   \\[1mm] [e_1,e_3]=\delta (\alpha_{n-2}e_{n-1}+\alpha_{n-1}e_n),& [e_1,e_i]=0, &5\leq i\leq n.\end{array}$$

Now we check the identity (\ref{eq:dualp}) for the triple $\{e_1, e_1, e_3\}$:
$$\delta e_1\cdot[e_1,e_3]=[e_2,e_3]+[e_1,e_4].$$

From this, we derive the following product
$$[e_2,e_3]=\frac{\delta^2-\delta}{2}\alpha_{n-2}e_n.$$

Applying induction by $i+j$, we prove $[e_i,e_j]=0$ for $i+j\geq 6.$

Now, we consider the condition (\ref{eq:dualp}) for the  triple  $\{e_1, e_i, e_j\}$:
$$0=\delta e_i  \cdot [e_1,e_j]= [e_i\cdot e_1,e_j]+[e_1,e_i \cdot e_j]=[e_{i+1},e_j].$$

So we get $[e_{i},e_j]=0$ for $i+j\geq 6$. Thus, we get the proof of the theorem.

\end{proof}

The following theorem establishes a necessary and sufficient condition for two algebras in the family ${\bf TP}_{\delta}(\alpha_{n-2},\alpha_{n-1},  \alpha_n)$ to be isomorphic.

\begin{theorem}\label{M1} Let ${\bf TP}_{\delta}(\alpha_{n-2},\alpha_{n-1},\alpha_n)$  and ${\bf TP}'_{\delta}(\alpha_{n-2}',\alpha_{n-1}',\alpha_n')$ are isomorphic algebras. Then there exists an automorphism $\varphi$ between these algebras such that:
$$\begin{array}{l}
     \alpha_{n-2}'=\frac{\alpha_{n-2}}{A_1^{n-5}}, \quad \alpha_{n-1}'=\frac{A_1\alpha_{n-1}+A_2\alpha_{n-2}(2\delta-n+2)}{A_1^{n-3}},\\[1mm]
\alpha_n'=\frac{2\alpha_nA_1^2+2\alpha_{n-1}A_1A_2(2 \delta - n+1)+\alpha_{n-2}(A_1A_3(\delta^2+3\delta-2n+4)+A_2^2(3\delta^2+\delta(3-4n)+n^2-n-2))}{2A_1^{n-1}}.\end{array}$$
\end{theorem}
\begin{proof}
Using the automorphism of the algebra $\mu_0^n$ from Theorem \ref{3.1}, we introduce the following notations:
$$e_1'=\sum\limits_{i=1}^{n} A_ie_i, \quad e_2'=\sum\limits_{i=2}^{n}\sum\limits_{k_1+k_2=i} A_{k_1}A_{k_2}e_i,$$
$$e_{n-2}'=A_1^{n-2}e_{n-2}+(n-2)A_1^{n-3}A_2e_{n-1}+\left(\frac{(n-2)(n-3)}{2}A_1^{n-4}A_2^2+(n-2)A_1^{n-3}A_3\right)e_n,$$
$$e_{n-1}'=A_1^{n-1}e_{n-1}+(n-1)A_1^{n-2}A_2e_n, \quad e_n'=A_1^ne_n.$$
Thus, we consider
$$[e_1',e_2']=\alpha_{n-2}'e_{n-2}'+\alpha_{n-1}'e_{n-1}'+\alpha_n'e_n'$$
$$=\alpha_{n-2}'\left(A_1^{n-2}e_{n-2}+(n-2)A_1^{n-3}A_2e_{n-1}+\left(\frac{(n-2)(n-3)}{2}A_1^{n-4}A_2^2+(n-2)A_1^{n-3}A_3\right)e_n\right)$$
$$+\alpha_{n-1}'(A_1^{n-1}e_{n-1}+(n-1)A_1^{n-2}A_2e_n)+ \alpha_n'A_1^ne_n.$$
On the other hand, we have

$$ [e_1',e_2']=[\sum\limits_{i=1}^{n} A_ie_i,\sum\limits_{i=2}^{n}\sum\limits_{k_1+k_2=i} A_{k_1}A_{k_2}e_i]$$
$$=A_1^3(\alpha_{n-2}e_{n-2}+\alpha_{n-1}e_{n-1}+\alpha_ne_n)+2\delta A_1^2A_2(\alpha_{n-2}e_{n-1}+\alpha_{n-1}e_n)$$
$$+\frac{\delta^2+\delta}{2}(2A_1^2A_3+A_1A_2^2)\alpha_{n-2}e_n+\frac{\delta^2-\delta}{2}(2A_1A_2^2-A_1^2A_3)\alpha_{n-2}e_n$$$$=
A_1^3\alpha_{n-2}e_{n-2}+(A_1^3\alpha_{n-1}+2\delta A_1^2A_2\alpha_{n-2})e_{n-1}$$
$$+\left(A_1^3\alpha_n+2\delta A_1^2A_2\alpha_{n-1}+\frac{\delta^2+\delta}{2}(2A_1^2A_3+A_1A_2^2)\alpha_{n-2}+\frac{\delta^2-\delta}{2}(2A_1A_2^2-A_1^2A_3)\alpha_{n-2}\right)e_n.$$

By comparing the coefficients of the corresponding basic elements, we obtain the expressions given in the lemma for $\alpha_{n-2}'$, $\alpha_{n-1}'$, and $\alpha_n'$.
\end{proof}

\begin{theorem} Let $(\mu_0^n, \cdot, [-,-])$ be a transposed $\delta$-Poisson algebra structure defined on the associative  algebra $\mu_0^n$ and $n\geq5$. Then this algebra is isomorphic to one of the following pairwise non-isomorphic algebras:
$${\bf TP}_{\delta}(1,0,0), \ {\bf TP}_{\delta}(0,1,0), \  {\bf TP}_{\delta}(0,0,1),\  $$
$${\bf TP}_{\frac{n-2}{2}}(1,\alpha,0), \
{\bf TP}_{\delta^2+3\delta=2n-4}(1,0,\alpha),\ {\bf TP}_{\frac{n-1}{2}}(0,1,\alpha), \ \alpha\in\mathbb{C}^{*}.$$
\end{theorem}
\begin{proof} Suppose we are given a transposed $\delta$-Poisson algebra ${\bf TP_\delta}(\alpha_{n-2},\alpha_{n-1}, \alpha_n)$. If we perform a general substitution  on this algebra, then from Theorem \ref{M1} we have the following equalities.

$$\begin{array}{l}
     \alpha_{n-2}'=\frac{\alpha_{n-2}}{A_1^{n-5}}, \quad \alpha_{n-1}'=\frac{A_1\alpha_{n-1}+A_2\alpha_{n-2}(2\delta-n+2)}{A_1^{n-3}},\\[1mm]
\alpha_n'=\frac{2\alpha_nA_1^2+2\alpha_{n-1}A_1A_2(2 \delta - n+1)+\alpha_{n-2}(A_1A_3(\delta^2+3\delta-2n+4)+A_2^2(3\delta^2+\delta(3-4n)+n^2-n-2))}{2A_1^{n-1}}.\end{array}$$

We can say that $\alpha_{n-2}$ is invariant for both zero and non-zero values, so we can divide $\alpha_{n-2}$ into the following cases based on the zero and non-zero values.

\begin{enumerate}
    \item[(1)] Let $\alpha_{n-2}\neq0$. Then we choose $A_1=\sqrt[n-5]{\alpha_{n-2}}$ and we have $\alpha_{n-2}'=1$. We apply the automorphism in Theorem \ref{3.1} once again to obtain the following  expressions:
$$\alpha_{n-1}'=\alpha_{n-1}+A_2(2\delta-n+2),$$
$$\alpha_n'=\frac{2\alpha_n+2\alpha_{n-1}A_2(2 \delta - n+1)+A_3(\delta^2+3\delta-2n+4)+A_2^2(3\delta^2+\delta(3-4n)+n^2-n-2)}{2}.$$
We obtain the following possible cases

\begin{enumerate}
    \item If $\delta=\frac{n-2}{2},$ then we derive
    $$\alpha_{n-1}'=\alpha_{n-1}, \quad
 \alpha_n'=
 \frac{8\alpha_n-8\alpha_{n-1}A_2+A_3(n^2-6n+8)-A_2(n^2-6n+8)}{8}.$$ Then, by choosing $A_2=0, \ A_3=-\frac{8\alpha_n}{n^2-6n+8},$ we have the alebra ${\bf TP}_{\frac{n-2}{2}}(1,\alpha,0)$.

\item If $\delta\neq\frac{n-2}{2},$ then, by putting $A_2=\frac{\alpha_{n-1}}{n-2-2\delta}$, we conclude that $\alpha_{n-1}'=0$. We once again use the automorphism in Theorem \ref{3.1} to obtain the following expression:
$$\alpha_n'=\frac{2\alpha_n+A_3(\delta^2+3\delta-2n+4)}{2}.$$

If $\delta^2+3\delta=2n-4$ then we have the algebra ${\bf TP}_{\delta^2+3\delta=2n-4}(1,0,\alpha)$.

If $\delta^2+3\delta\neq2n-4$ then we obtain the algebra ${\bf TP}_{\delta^2+3\delta\neq 2n-4}(1,0,0)$.
\end{enumerate}

 \item[(2)] Let $\alpha_{n-2}=0$.  Then, we get

$$\alpha_{n-2}'=0, \quad \alpha_{n-1}'=\frac{\alpha_{n-1}}{A_1^{n-4}}, \quad \alpha_n'=\frac{2\alpha_nA_1+2\alpha_{n-1}A_2(2 \delta - n+1)}{2A_1^{n-2}}.$$
In this case, $\alpha_{n-1}$  is invariant for zero and non-zero values, so we obtain the following possible cases.
\begin{enumerate}
     \item  $\alpha_{n-1}\neq 0$.
     Then, we choose $A_1=\sqrt[n-4]{\alpha_{n-1}}$ and we have $\alpha_{n-1}'=1$. We again apply the automorphism in Theorem \ref{3.1} to obtain the following expression:
$$\alpha_n'=\alpha_n+A_2(2 \delta - n+1).$$

If $\delta=\frac{n-1}{2}$, then we have ${\bf TP}_{\frac{n-1}{2}}(0,1,\alpha)$.
 
If $\delta\neq\frac{n-1}{2}$, then we have ${\bf TP}_{\delta\neq\frac{n-1}{2}}(0,1,0)$.

\item  If $\alpha_{n-1}=0$, then we derive $\alpha_n'=\frac{\alpha_n}{A_1^{n-3}}$ and choosing $A_1=\sqrt[n-3]{\alpha_{n}}$, we obtain the algebra ${\bf TP}_{\delta}(0,0,1)$.
\end{enumerate}
\end{enumerate}
\end{proof}

In the following theorems, we consider 2,3 and 4-dimensional cases. For $\delta\neq0$, in \cite{Bai} proved that any 2-dimensional complex transposed $\delta$-Poisson algebra on the 2-dimensional associative algebra whose product is given by $e_1\cdot e_1=e_2$ is isomorphic to one of the following transposed Poisson algebras:
$${\bf TP}_{\delta}(0,0): \ e_1\cdot e_1=e_2; \quad {\bf TP}_{\delta}(0,1): \ e_1\cdot e_1=e_2, \ [e_1,  e_2]=e_2, \ \delta\in \mathbb{C}\setminus\{0\}.$$

If $\delta = 0$, then we have a transposed $0$-Poisson algebra with the multiplication  
\[{\bf TP}_0: \ e_1 \cdot e_1 = e_2, \ [e_1, e_2] = \alpha_1 e_1 + \alpha_2 e_2.
\]  
Then, by applying the Theorem \ref{3.1}, we obtain the following expressions:
$$\alpha_1'=A_1^2\alpha_1, \ \alpha_2'=A_1\alpha_2-A_2\alpha_1.$$
By considering the possible values of the parameters $\alpha_1$ and $\alpha_2$, we obtain the following transposed $0$-Poisson algebras:
$${\bf TP}_{0}(0,0): \ e_1\cdot e_1=e_2; \quad {\bf TP}_{0}(0,1): \ e_1\cdot e_1=e_2, \ [e_1,  e_2]=e_2, \quad \ {\bf TP}_{0}(1,0): \ e_1\cdot e_1=e_2, \ [e_1,  e_2]=e_1.$$
Thus, we have obtained the following two-dimensional transposed $\delta$-Poisson algebras.
$${\bf TP}_{\delta}(0,0), \ {\bf TP}_{\delta}(0,1), \ {\bf TP}_{0}(1,0),\  \delta\in \mathbb{C}.$$

Now we consider the transposed $\delta$-Poisson algebra structures on the three-dimensional   null-filiform associative algebra $\mu_0^3$. In this case, we set 
$$[e_1,e_2]=\alpha_1 e_1+\alpha_2 e_2+\alpha_3 e_3.$$

By analyzing the identity (\ref{eq:dualp}) for the triples $\{e_1, e_1, e_2\}, \ \{e_1, e_1, e_3\} , \ \{e_2, e_1, e_2\}$ and also by considering Jacobi identity for triple $\{e_1, e_2, e_3\}$ we have 
\begin{equation}\label{dim=3}
    [e_1,e_2]=\alpha_1 e_1+\alpha_2 e_2+\alpha_3 e_3, \ [e_1,e_3]=\delta\alpha_{2}e_3, \ [e_2,e_3]=0, \ \delta\alpha_1=0.\end{equation}

\begin{theorem} Let $(\mu_0^3, \cdot, [-,-])$ be a transposed $\delta$-Poisson algebra structure defined on $\mu_0^3$ and $\delta\neq0$. Then it is isomorphic to one of the following pairwise non-isomorphic algebras:
\[{\bf TP}_{\delta}(0,1,0), \ {\bf TP}_{\delta}(0,0,\alpha),  \ {\bf TP}_{1}(0,1,\alpha), \ \alpha\in\mathbb{C}.\]
\end{theorem}

\begin{proof} Let $(\mu_0^3, \cdot, [-,-])$ be a transposed $\delta$-Poisson algebra and $\delta\neq0$. Then, according to (\ref{dim=3}), we have the following
$$ {\bf TP}(0,\alpha_2,\alpha_3):
\ \begin{cases}
e_1\cdot e_1=e_2, \ e_1\cdot e_2=e_3, \ e_2\cdot e_1=e_3,\\[1mm]
[e_1,e_2]=\alpha_{2}e_2+\alpha_{3}e_3, \ [e_1,e_3]=\delta\alpha_{2}e_3.\\[1mm]
 \end{cases}$$
By applying the automorphism in Theorem \ref{3.1}, we get the following relations:
$$\alpha_2'=A_1\alpha_2, \ \alpha_3'=\frac{A_1\alpha_3+(2\delta-2)A_2\alpha_2}{A_1}.$$

\begin{itemize}
  \item[(1)] Let $\delta=1.$ Then, we have $\alpha_2'=A_1\alpha_2, \ \alpha_3'=\alpha_3.$ In this case we obtain the algebras ${\bf TP}_{1}(0,1,\alpha)$ and ${\bf TP}_{1}(0,0,\alpha)$, where $\alpha\in\mathbb{C}.$
  \item[(2)] Let $\delta\neq1.$ Then:

If $\alpha_2\neq 0$, then by choosing $A_1=\frac{1}{\alpha_2}, \ A_2=-\frac{\alpha_3}{(2\delta-2)\alpha_2^2},$
    we have the algebra ${\bf TP}_{\delta}(0,1,0)$.

If $\alpha_2=0$, then $\alpha_3'=\alpha_3$ and we derive the algebra ${\bf TP}_{\delta}(0,0,\alpha).$
\end{itemize}

\end{proof}

\begin{theorem} Let $(\mu_0^3, \cdot, [-,-])$ be a transposed $0$-Poisson algebra structure defined on $\mu_0^3$. Then, it is isomorphic to one of the following pairwise non-isomorphic algebras:
\[{\bf TP}_{0}(1,0,0), \ {\bf TP}_{0}(0,1,0),  \ {\bf TP}_{0}(0,0,\alpha), \ \alpha\in\mathbb{C}.\]
\end{theorem}

\begin{proof} Let $(\mu_0^3, \cdot, [-,-])$ be a transposed $0$-Poisson algebra. Then, according to (\ref{dim=3}), we have the following
$$ {\bf TP}(\alpha_1,\alpha_2,\alpha_3):
\ \begin{cases}
e_1\cdot e_1=e_2, \ e_1\cdot e_2=e_3, \ e_2\cdot e_1=e_3,\\[1mm]
[e_1,e_2]=\alpha_1e_1+\alpha_{2}e_2+\alpha_{3}e_3.\\[1mm]
 \end{cases}$$
By applying the automorphism in Theorem \ref{3.1}, we get the following relations:
$$\alpha_1'=A_1^2\alpha_1, \ \alpha_2'=A_1\alpha_2-A_2\alpha_1, \ \alpha_3'=\frac{A_1^2\alpha_3+(2A_2^2-A_1A_3)\alpha_1-2A_1A_2\alpha_2}{A_1^2}.$$

By checking all possible cases of the parameters $\alpha_1, \alpha_2$ and $\alpha_3$, we obtain the above algebras.
\end{proof}

\begin{corollary} Let $(\mu_0^3, \cdot, [-,-])$ be a transposed $\delta$-Poisson algebra structure defined on $\mu_0^3$. Then it is isomorphic to one of the following pairwise non-isomorphic algebras:
\[{\bf TP}_{0}(1,0,0), \ {\bf TP}_{\delta}(0,1,0), \ {\bf TP}_{\delta}(0,0,\alpha),  \ {\bf TP}_{1}(0,1,\alpha), \ \alpha\in\mathbb{C}.\]
\end{corollary}

Now we consider the transposed $\delta$-Poisson algebra structures on the four-dimensional   null-filiform associative algebra $\mu_0^4$. In this case, we set 
$$[e_1,e_2]=\alpha_1 e_1+\alpha_2 e_2+\alpha_3 e_3+\alpha_4 e_4.$$

By analyzing the identity (\ref{eq:dualp}) for the triples 
$$\{e_1, e_1, e_2\}, \ \{e_1, e_1, e_3\} , \ \{e_2, e_1, e_2\}, \ \{e_1, e_1, e_4\}, \{e_3, e_1, e_2\}$$ we have $\delta\alpha_1=0$ and 
$${\bf TP}(\alpha_1, \alpha_{2},\alpha_{3},  \alpha_4):
\begin{cases}
e_i\cdot e_j=e_{i+j},   &2\leq i+j \leq 4, \\[1mm]
[e_1,e_2]=\alpha_{1}e_{1}+\alpha_{2}e_{2}+\alpha_{3}e_{3}+\alpha_3e_3, \\[1mm]
[e_1,e_3]=\delta (\alpha_{2}e_{3}+\alpha_{3}e_4), \\[1mm] [e_2,e_3]=\frac{\delta^2-\delta}{2}\alpha_{2}e_4, \\[1mm] [e_1,e_4]=\frac{\delta^2+\delta}{2}\alpha_{2}e_4. \\[1mm]
\end{cases}$$

\begin{theorem}\label{deltadim4} Let $(\mu_0^4, \cdot, [-,-])$ be a transposed $\delta$-Poisson algebra structure defined on $\mu_0^4$ and $\delta\neq0$. Then it is isomorphic to one of the following pairwise non-isomorphic algebras:
$${\bf TP}_{1}(0, 1,\alpha,0), \ {\bf TP}_{1}(0, 1,0, \alpha),$$
\[{\bf TP}_{\delta}(0,1,0,0), \ {\bf TP}_{\delta}(0,0,\alpha,0), {\bf TP}_{\delta}(0,0,0,1), {\bf TP}_{-4}(0,1,0,\alpha),
{\bf TP}_{\frac{3}{2}}(0,0,\alpha,1)\ \alpha\in\mathbb{C}.\]
\end{theorem}

\begin{proof} Let $(\mu_0^4, \cdot, [-,-])$ be a transposed Poisson algebra. According to the above, this algebra is isomorphic to the algebra ${\bf TP}_{\delta}(0, \alpha_2,\alpha_3,\alpha_4)$. Applying a change of basis of Theorem \ref{3.1}, we get the following relations:
$$\alpha_2'=A_1\alpha_2, \ \alpha_3'=\frac{A_1\alpha_3+(2\delta-2)A_2\alpha_2}{A_1}, $$
$$\alpha_4'=\frac{2A_1^2 \alpha_4+2(2\delta-3) A_1 A_2 \alpha_3+((\delta^2+3\delta-4)A_1A_3+(3\delta^2-13\delta+10)A_2^2)\alpha_2}{2A_1^3}.$$
\begin{enumerate}
    \item[(1)] Let $\delta=1$. Then "
    $$\alpha_2'=A_1\alpha_2, \ \alpha_3'=\alpha_3, \ \alpha_4'=\frac{A_1\alpha_4-A_2\alpha_3}{A_1^2}.$$
\begin{enumerate}
  \item If $\alpha_2\neq0$ and $\alpha_3\neq0,$ then by putting $A_1=\frac{1}{\alpha_2}, \ A_2=-\frac{\alpha_4}{\alpha_3},$ we have the algebra ${\bf TP}_{1}(0, 1,\alpha,0).$  
 \item If $\alpha_2\neq0$ and $\alpha_3=0,$ then by putting $A_1=\frac{1}{\alpha_2},$ we obtain the algebra ${\bf TP}_{1}(0, 1,0, \alpha).$ 
 \item If $\alpha_2=0$ and $\alpha_3\neq0,$ then by choosing $A_2=\frac{A_1\alpha_4}{\alpha_3},$ we obtain the algebra ${\bf TP}_{1}(0, 0, \alpha, 0).$ 
 \item If $\alpha_2=\alpha_3=0$ and $\alpha_4\neq0,$ then by choosing $A_1=\alpha_4,$ we obtain the algebra ${\bf TP}_{1}(0, 0, 0, 1).$ 
 \item If $\alpha_2=\alpha_3=\alpha_4=0,$ then we have the algebra ${\bf TP}_{1}(0, 0, 0, 0).$
\end{enumerate}

  \item[(2)] Let $\delta\neq1$. Then:
    \begin{enumerate}
    \item $\alpha_2\neq 0$, then by choosing $A_1=\frac{1}{\alpha_2}, \ A_2=-\frac{\alpha_3}{(2\delta-2)\alpha_2^2},$
    we have  $\alpha_2'=1,\ \alpha_3'=0.$ We apply the automorphism in Theorem \ref{3.1} once again to obtain the following:
    $$\alpha_4'=\frac{2\alpha_4+(\delta^2+3\delta-4)A_3}{2}.$$
    
    If $\delta\neq-4$, we put $A_3=-\frac{2\alpha_4}{\delta^2+3\delta-4},$ then we obtain the algebra ${\bf TP}_{\delta\neq -4}(0, 1,0,0)$;

   If $\delta=-4$, then we derive the algebra ${\bf TP}_{-4}(0,1,0,\alpha)$;

    \item Let $\alpha_2=0$ and $\alpha_3\neq 0.$ Then, we have
    $$\alpha_3'=\alpha_3, \ \alpha_4'=\frac{A_1\alpha_4+(2\delta-3)A_2 \alpha_3}{A_1^2}.$$

If $\delta\neq\frac{3}{2}$, choosing $A_1=1, A_2=\frac{\alpha_4}{(3-2\delta)\alpha_3}$ then we derive the algebra ${\bf TP}_{\delta\neq \frac{3}{2}}(0, 0,\alpha,0),$;

If $\delta=\frac{3}{2}$, then we have
    $$\alpha_3'=\alpha_3, \ \alpha_4'=\frac{\alpha_4}{A_1}.$$
For zero or nonzero values of parameter $\alpha_4$, we derive the algebras ${\bf TP}_{\frac{3}{2}}(0,0,\alpha,0)$ and  ${\bf TP}_{\frac{3}{2}}(0,0,\alpha,1)$, respectively.

\item Let $\alpha_2=\alpha_3=0$ and $\alpha_4\neq 0$. Then, by setting $A_1=\alpha_4$, we obtain the algebra ${\bf TP}_{\delta}(0,0,0,1)$;
    \item If $\alpha_2=\alpha_{3}=\alpha_4=0,$ then the algebra simplifies to ${\bf TP}_{\delta}(0,0,0,0)$.
    \end{enumerate}
\end{enumerate}
\end{proof}

Using the Theorems \ref{tdelta=0} and \ref{deltadim4}, we derive the classification four-dimensional complex transposed $\delta$-Poisson algebra associated on $\mu_0^4.$

\begin{theorem} Let $(\mu_0^4, \cdot, [-,-])$ be a transposed $\delta$-Poisson algebra structure defined on the associative  algebra $\mu_0^n$. Then, this algebra is isomorphic to one of the following pairwise non-isomorphic algebras:
\[{\bf TP}_0(1,0,0,0),\ {\bf TP}_{1}(0, 1,\alpha,0), \ {\bf TP}_{1}(0, 1,0, \alpha),\]
\[{\bf TP}_{\delta}(0,1,0,0), \ {\bf TP}_{\delta}(0,0,\alpha,0), {\bf TP}_{\delta}(0,0,0,1), {\bf TP}_{-4}(0,1,0,\alpha),
{\bf TP}_{\frac{3}{2}}(0,0,\alpha,1)\ \alpha\in\mathbb{C}.\]
\end{theorem}

All Poisson algebra structures and the trivial Poisson algebra on null-filiform associative algebras have been obtained in \cite{AFM}. Below, we construct all $\delta$-Poisson algebra structures on null-filiform associative algebras.

\begin{theorem} Let $(\mu_0^n, \cdot, [-,-])$ be a $\delta$-Poisson algebra structure defined on the associative algebra $\mu_0^n$. Then, $(\mu_0^n,\cdot, [-,-])$ is a trivial $\delta$-Poisson algebra. \end{theorem}

\begin{proof} Let $(\mu_0^n, \cdot, [-,-])$ be a $\delta$-Poisson algebra. To establish the table of multiplications for the operation $[-,-]$ in this Poisson algebra structure, we
consider the following computation for $1\leq i \leq n-1$:
$$[e_1, e_{i+1}] =[e_1, e_1\cdot e_i
] = \delta([e_1, e_1]\cdot e_i + e_1\cdot[e_1, e_i
]) = \delta e_1\cdot [e_1, e_i
].$$
From this we get $[e_1, e_i
] = 0,$ for $2\leq i\leq n.$

Next, considering the following equalities
$$[e_i, e_2] = [e_i, e_1\cdot e_1] = \delta([e_i, e_1]\cdot e_1 + e_1\cdot  [e_i, e_1]) = 0, \  3\leq i\leq n-1,$$
$$[e_i, e_j] = [e_i, e_1\cdot e_{j-1}] = \delta([e_i, e_1]\cdot e_{j-1} + e_1\cdot [e_i, e_{j-1}]) = 0, \ 3\leq i, j\leq n,$$
we obtain
$$[e_i, e_j] = 0,\ 1\leq i, j\leq n.$$\end{proof}


\begin{thebibliography}{99}

\bibitem{AFM} Abdelwahab H., Fern\'{a}ndez Ouaridi A., Mart\'{\i}n Gonz\'{a}lez C., {\it Degenerations of Poisson algebras,}
Journal of Algebra and Its Applications,
 24 (2025), 3, 2550087.


\bibitem{AKS} 
 Abdelwahab H., Kaygorodov I., Sartayev B., 
 {\it$\delta$-Poisson and transposed $\delta$-Poisson  algebras,} arXiv:2411.05490.



\bibitem{AAE} Abdurasulov K., Adashev J., Eshmeteva S., {\it  Transposed Poisson structures on solvable Lie algebras with filiform nilradical,} Communications in Mathematics, 32 (2024), 3, 441-483.

\bibitem{ADSS} Abdurasulov K., Deraman F., Saydaliyev A., Sapar S.H., {\it Transposed Poisson structures on low-dimensional quasi-filiform Lie algebras of maximum length,} Lobachevskii Journal of Mathematics, 45 (2024), 11, 5735-5749.



  \bibitem{ABT} Adashev J., Berdalova X., Toshtemirova F., {\it Transposed Poisson algebra structures on null-filiform associative algebras,} arXiv:2503.06295.
    

\bibitem{aku} Arzikulov F., Karimjanov I., Umrzaqov S., {\it Local and 2-local automorphisms of null-filiform and filiform associative algebras,}  Journal of Algebra Combinatorics Discrete Structures and Applications, 11 (2024), 3, 151-164.

\bibitem{Bai} Bai C., Bai R., Guo L., Wu Y., {\it Transposed Poisson algebras, Novikov-Poisson algebras, and 3-Lie algebras,} Journal
of Algebra, 632 (2023), 535-566.




		\bibitem{bfk23}   Beites P.,    Fern\'andez Ouaridi A.,   Kaygorodov I.,
		{\it The algebraic and geometric classification of transposed Poisson algebras,} Revista de la Real Academia de Ciencias Exactas, F\'{\i}sicas y Naturales. Serie A. Matem\'{a}ticas,   117 (2023), 2,   55.


		
		\bibitem{bfk22} Beites P., Ferreira B. L. M., Kaygorodov I., {\it Transposed Poisson structures,} Results in Mathematics, 79 (2024), 93.

\bibitem{fm25} Fairon M., {\it Modified double brackets and a conjecture of S. Arthamonov,} Communications in Mathematics, 33
(2025), 3, 5.

\bibitem{FKL}
		Ferreira B. L. M., Kaygorodov I., Lopatkin V., {\it $\frac{1}{2}$-derivations of Lie algebras and transposed Poisson algebras,} 		Revista de la Real Academia de Ciencias Exactas, F\'{\i}sicas y Naturales. Serie A. Matem\'{a}ticas, 115 (2021), 3, 142.

	 \bibitem{fer23} Fern\'{a}ndez Ouaridi A., {\it On the simple transposed Poisson algebras and Jordan superalgebras,} Journal of Algebra, 641 (2024), 173-198.

 \bibitem{fil98} Filippov V., {\it $\delta$-Derivations of Lie algebras,} Siberian Mathematical Journal, 39 (1998), 6, 1218-1230.
		

\bibitem{jawo}
		Jaworska-Pastuszak A., Pogorza{\l}y Z.,
	{\it	Poisson structures for canonical algebras},
		Journal of Geometry and Physics, 148 (2020), 103564.

\bibitem{k25} Kaygorodov I., {\it        $\delta$-Novikov and $\delta$-Novikov-Poisson algebras}, arXiv:2505.08043.



\bibitem{k23} Kaygorodov I., {\it Non-associative algebraic structures: classification and structure,}
Communications in Mathematics, 32 (2024), 3, 1-62.

		
		\bibitem{kk21}
		Kaygorodov I., Khrypchenko M., {\it Poisson structures on finitary incidence algebras,}
		Journal of  Algebra, 578 (2021), 402-420.
		
		
		\bibitem{kk22}       Kaygorodov I., Khrypchenko M.,
		{\it Transposed Poisson  structures on Block  Lie algebras and superalgebras,}
		Linear Algebra and Its Applications,  656 (2023), 167-197.
		
		\bibitem{kk23}       Kaygorodov I., Khrypchenko M.,
		{\it Transposed Poisson structures on Witt type algebras,} 	Linear Algebra and its Applications,   665  (2023),  196-210.	
		
		
		\bibitem{kkg23}      Kaygorodov I.,  Khrypchenko M.,
		{\it Transposed Poisson structures on generalized Witt algebras and Block Lie algebras,}   Results in Mathematics, 78 (2023),  5,  186.
		
		\bibitem{KK7} Kaygorodov I., Khrypchenko M., {\it Transposed Poisson structures on the Lie algebra of upper triangular matrices,} Portugaliae Mathematica,  81 (2024), 1-2, 135-149.

\bibitem{kkinc}
Kaygorodov I.,  Khrypchenko M.,
{\it Transposed Poisson structures on   Lie incidence algebras}, Journal of Algebra, 647 (2024), 458-491.

\bibitem{KKh} Kaygorodov I., Khudoyberdiyev A., {\it Transposed Poisson structures on solvable and perfect Lie algebras,}  Journal of Physics A: Mathematical and Theoretical,  57  (2024), 035205.

\bibitem{KKhZ} Kaygorodov I., Khudoyberdiyev A., Shermatova Z., {\it Transposed Poisson structures on not-finitely graded Witt-type
algebras,} Bolet\'{\i}n de la Sociedad Matem\'{a}tica Mexicana, 31 (2025), 1, 22.

\bibitem{KKhZ1} Kaygorodov I., Khudoyberdiyev A., Shermatova Z., {\it Transposed Poisson structures on Virasoro-type algebras,} Journal
of Geometry and Physics, 207 (2025), 105356.

		\bibitem{klv22}    Kaygorodov I., Lopatkin V., Zhang Z.,
		{\it Transposed Poisson structures on  Galilean and solvable Lie algebras,}   Journal of Geometry and  Physics, 187 (2023),   104781.

\bibitem{Kh25} Khrypchenko M., {\it  $\sigma$-matching and interchangeable structures on certain associative algebras,} Communications in
Mathematics, 33 (2025), 3, 6.



\bibitem{lb23}
Liu G.,  Bai C.,
{\it  A bialgebra theory for transposed Poisson algebras via anti-pre-Lie
  bialgebras and anti-pre-Lie-Poisson bialgebras},
  Communications in Contemporary Mathematics, (2023), 2350050.

\bibitem{MO}  Dekimpe K., Ongenae V., {\it Filiform left-symmetric algebras,} Geometriae Dedicata, 74 (1999), 2, 165-199.


\bibitem{kms} Sartayev B., {\it Some generalizations of the variety of transposed Poisson algebras,} Communications in Mathematics, 32 (2024),  2, 55-62.

\bibitem{ytk} Yang Ya.,  Tang X.,  Khudoyberdiyev A., {\it Transposed Poisson structures on Schrodinger algebra in $(n+1)$-dimensional space-time,} arXiv:2303.08180.


\bibitem{YYZ07}
		Yao Y., Ye Y., Zhang P.,
		{\it Quiver Poisson algebras},
		Journal of  Algebra, 312 (2007), 2, 570-589.

		 

		\bibitem{yh21}
		Yuan L.,  Hua Q.,
		{\it $\frac{1}{2}$-(bi)derivations and transposed Poisson algebra structures on Lie algebras,}
		Linear and Multilinear Algebra, 70 (2022), 22, 7672-7701.

\bibitem{zz} Zohrabi A., Zusmanovich P.,
  A $\delta$-first Whitehead lemma for Jordan algebras,
  Communications in Mathematics,  33 (2025), 1, 2.


\end{thebibliography}
\end{document}